\documentclass[a4paper, reqno, 11pt]{amsart}

\usepackage{amscd, amssymb, amsthm, amsmath, color, cases}
\usepackage{graphicx, psfrag, enumerate, multirow, algorithmic}
\usepackage[linesnumbered,ruled,vlined]{algorithm2e}
\usepackage{changepage}

\newtheorem{lemma}{Lemma}
\newtheorem{theorem}[lemma]{Theorem}

\newtheorem{proposition}[lemma]{Proposition}

\newtheorem{remark}{Remark}

\newcommand*\xbar[1]{%
   \hbox{%
     \vbox{%
       \hrule height 0.5pt 
       \kern0.5ex
       \hbox{%
         \ensuremath{#1}%
       }%
     }%
   }%
}

\begin{document}
\title{The Undirected Optical Indices of Trees}

\author{Yuan-Hsun Lo}
\address{Department of Applied Mathematics, National Pingtung University, Pingtung 900, Taiwan, ROC}
\email[Y.-H.~Lo]{yhlo@mail.nptu.edu.tw}

\author{Hung-Lin Fu}
\address{Department of Applied Mathematics, National Yang Ming Chiao Tung University, Hsinchu 300, Taiwan, ROC}
\email[H.-L.~Fu]{hlfu@math.nctu.edu.tw}

\author{Yijin Zhang}
\address{School of Electronic and Optical Engineering, Nanjing University of Science and Technology, Nanjing, China} 
\email[Y.~Zhang]{yijin.zhang@gmail.com}

\author{Wing Shing Wong}
\address{Department of Information Engineering, the Chinese University of Hong Kong, Shatin, Hong Kong}
\email[W.~S.~Wong]{wswong@ie.cuhk.edu.hk}





\maketitle

\begin{abstract}
For a connected graph $G$, an instance $I$ is a set of pairs of vertices and a corresponding routing $R$ is a set of paths specified for all vertex-pairs in $I$.
Let $\mathfrak{R}_I$ be the collection of all routings with respect to $I$.
The undirected optical index of $G$ with respect to $I$ refers to the minimum integer $k$ to guarantee the existence of a mapping $\phi:R\to\{1,2,\ldots,k\}$, such that $\phi(P)\neq\phi(P')$ if $P$ and $P'$ have common edge(s), over all routings $R\in\mathfrak{R}_I$.
A natural lower bound of the undirected optical index is the edge-forwarding index, which is defined to be the minimum of the maximum edge-load over all possible routings.
Let $w(G,I)$ and $\pi(G,I)$ denote the undirected optical index and edge-forwarding index with respect to $I$, respectively.
In this paper, we derive the inequality $w(T,I_A)<\frac{3}{2}\pi(T,I_A)$ for any tree $T$, where $I_A:=\{\{x,y\}:\,x,y\in V(T)\}$ is the all-to-all instance.
\end{abstract}

\section{Introduction} \label{sec:intro}


Let $G$ be a connected graph with vertex set $V(G)$ and edge set $E(G)$.
An \emph{instance} $I$ is a set (or, multiset) of vertex-pairs of $V(G)$.
A \emph{routing} $R$ in $G$ with respect to $I$ is a set of $|I|$ paths, one for each vertex-pair in $I$.
That is, $\{x,y\}\in I$ if and only if there is a path having $x$ and $y$ as its terminal vertices.
Such a path is denoted by $P_{x,y}$ or $P_{y,x}$.
A \emph{$k$-path-coloring} of $R$ is a mapping $\phi:R\to\{1,2,\ldots,k\}$, and is said to be \emph{proper} if $\phi(P)\neq\phi(P')$ whenever $P$ and $P'$ have common edge(s).
Let $w(G,I,R)$ be the minimum integer $k$ to guarantee the existence of a proper $k$-path-coloring of $R$.
Let $\mathfrak{R}_I$ denote the collection of all routings in $G$ with respect to $I$.
The \emph{undirected optical index} (or the \emph{path-chromatic number}) of $G$ with respect to $I$ is then defined to be 
\begin{align*}
w(G,I):=\min_{R\in\mathfrak{R}_I}w(G,I,R).
\end{align*}
Note that by constructing a graph $Q(R)$, say \emph{conflict graph}, on $R$ by paths $P$ and $P'$ being adjacent if and only if they have common edge(s), the value $w(G,I,R)$ turns out to be the chromatic number of $Q(R)$, i.e., $\chi(Q(R))$.

For a routing $R\in\mathfrak{R}_I$ and an edge $e\in E(G)$, the \emph{edge-load} of $e$, denoted by $\ell_{G,R}(e)$, is the number of paths in $R$ passing through $e$.
Let $\pi(G,I,R)$ denote the maximum value of $\ell_{G,R}(e)$ by going through all edges in $G$, i.e., $\pi(G,I,R)=\max_{e\in E(G)}\ell_{G,R}(e)$.
The \emph{edge-forwarding index} of $G$ with respect to $I$ is then defined by
\begin{align*}
\pi(G,I):=\min_{R\in\mathfrak{R}_I} \pi(G,I,R).
\end{align*}
By definition, the edge forwarding index provides a natural lower bound of the undirected optical index, namely, $\pi(G,I)\leq w(G,I)$ for any connected graph $G$ and instance $I$.

Analogous parameters can be introduced when considering a connected bidirected graph, which is a digraph obtained from a connected (undirected) graph by putting two opposite arcs on each edge.
In a bidirected graph $G$, an instance $I$ consists of ordered pairs of vertices and a corresponding routing $\vec{R}_I$ refers to a set of $|I|$ dipaths specified for all ordered pairs in $I$.
The \emph{optical index} and \emph{arc-forwarding index}, denoted by $\vec{w}(G,I)$ and $\vec{\pi}(G,I)$ respectively, are defined accordingly.
We use the right-arrow symbol to emphasize that the parameters are considered in a directed version.
It is worth noting that the evaluation of optical indices is known as the \emph{routing and wavelength assignment} (RWA) problem, which arises from the investigation of optimal wavelength allocation in an optical network that employs Wavelength Division Multiplexing (WDM)~\cite{Alexander93}.

For an arbitrary instance $I$, to evaluate the exact value of $w(G,I)$ has been shown to be NP-hard, even for trees~\cite{GJ85} and cycles~\cite{EJ01}.
Some approximation algorithms were proposed in~\cite{Tarjan85,RU94}.
The best known results are with approximation ratio $\frac{4}{3}$ for trees~\cite{EJ01} and approximation ratio $2-o(1)$ for cycles~\cite{Cheng04}.
When it comes to directed case, it is also NP-hard to determine $\vec{w}(G,I)$ for trees and cycles~\cite{EJ01}.
A $\frac{5}{3}$-approximation algorithm for trees was proposed in~\cite{EJKMP99} and a $2$-approximation algorithm for cycles was given in~\cite{Caragiannis09}.
As $\vec{\pi}(G,I)$ being a natural lower bound of $\vec{w}(G,I)$, Kaklamanis \emph{et al.}~\cite{KPEJ97} showed that $\frac{5}{3}\vec{\pi}(G,I)$ colors are enough when $G$ is a tree, and Tucker~\cite{Tucker75} showed that $2\vec{\pi}(G,I)-1$ colors are enough when $G$ is a cycle.
Interested readers are referred to \cite{BBGHPV97,GMZ15,HMS89,Lampis13,QZ04,Sole95} for more information.

Some literatures focused on the fundamental case when the instance consists of all vertex-pairs (or, ordered pairs of vertices for directed case), called \emph{all-to-all} instance and denoted by $I_A$.
That is, $|I_A|={|V(G)|\choose 2}$ for undirected case and $|I_A|=|V(G)|(|V(G)|-1)$ for directed case.
It has been proved that the equality $\vec{w}(G,I_A)=\vec{\pi}(G,I_A)$ holds for trees~\cite{GHP97}, cycles~\cite{Wilfong96}, trees of cycles~\cite{BPT99}, some Cartesian product of paths or cycles with equal lengths~\cite{Beauquier99,SSV97}, some certain compound graphs~\cite{ART01} and circulant graphs~\cite{ART01,GMZ15}.
Kosowski~\cite{Kosowski09} provided a family of graphs satisfying $\vec{w}(G,I_A)>\vec{\pi}(G,I_A)$.

The results for all-to-all instance on undirected case are relatively few.
The exact value of $w(G,I_A)$ and the gap between it to $\pi(G,I_A)$ are characterized for cycles~\cite{LZWF15} or complete $m$-ary trees~\cite{FLWZ17}.
Recently, it was conjectured in \cite{FLWZ17} that $w(G,I_A)$ is upper bounded by $\frac{3}{2}\pi(G,I_A)$ in the case when $G$ is a tree.
This paper is devoted to prove this conjecture.
It should be noted here that, both the $\frac{4}{3}$-approximation algorithm in \cite{EJ01} for undirected case and the method of the usage of $\frac{5}{3}\vec{\pi}(G,I)$ colors in \cite{KPEJ97} for directed case do not cover our result, because the former one does not take the gap between the undirected optical and edge-forwarding indices into consideration while the latter one focuses on directed trees and the ratio $\frac{5}{3}$ is larger than $\frac{3}{2}$, as claimed in this paper.

\section{Main Result}\label{sec:main}
Let $T$ be a tree.
There is a unique path to connect any pair of vertices in $T$, so $|\mathfrak{R}_I|=1$, for any instance $I$.
Hereafter we only consider the all-to-all instance and use $R$ to denote the unique all-to-all routing.
For convenience, $w(T,I_A)$, $\pi(T,I_A)$ and $\ell_{T,R}(e)$ are simply written as $w(T)$, $\pi(T)$ and $\ell_T(e)$, respectively.

Since each edge $e\in E(T)$ is a bridge, $\ell_T(e)$ is equal to the product of the numbers of vertices of the two components in $T-e$.
Therefore, a natural upper bound of $\pi(T)$ is obtained as follows.


\begin{proposition}\label{prop:cap-upper}
If $T$ is a tree of order $n$, then $$\pi(T) \leq \frac{n^2}{4}.$$ 
\end{proposition}

Here is the main result in this paper.

\begin{theorem}\label{thm:main}
If $T$ is a tree of order $n$ ($n\geq 2$), then
\begin{equation}\label{eq:main}
w(T)< \frac{3}{2}\pi(T).
\end{equation}
\end{theorem}

\begin{proof}
We proceed by induction on $n$.
It is obvious that \eqref{eq:main} holds when $n\leq 3$.
In what follows we consider $n\geq 4$, and assume \eqref{eq:main} holds for any tree of order less than $n$.

Let $\hat{e}$ be the edge maximizing the value $\ell_T(e)$ among all edges, that is, $\ell_T(\hat{e})=\pi(T)$.
Note that $\hat{e}$ may not be unique.
Let $A$ and $B$ be the two connected components of $T-\hat{e}$ with $a\geq b$, where $a:=|V(A)|$ and $b:=|V(B)|$.

We first consider the case when $b\geq\frac{3}{4}a$.
The paths in $R$ are partitioned into three classes.
\begin{itemize}
\item $\mathcal{P}_1:=\{P_{x,y}:\, x,y\in V(A)\}$.
\item $\mathcal{P}_2:=\{P_{x,y}:\, x,y\in V(B)\}$.
\item $\mathcal{P}_3:=\{P_{x,y}:\, x\in V(A) \text{ and } y\in V(B)\}$.
\end{itemize}
It follows from previous observation that $\chi(Q(\mathcal{P}_1))=w(A)$, $\chi(Q(\mathcal{P}_2))=w(B)$ and $\chi(Q(\mathcal{P}_3))=|\mathcal{P}_3|=a\cdot b$.
Observe that any two paths, one in $\mathcal{P}_1$ and another in $\mathcal{P}_2$, can receive the same color.
By the induction hypothesis and Proposition~\ref{prop:cap-upper}, it follows that
\begin{align*}
w(T) &\leq \max\left\{w(A),w(B)\right\} + |\mathcal{P}_3| \\ 
&< \max\big\{\frac{3}{2}\pi(A),\frac{3}{2}\pi(B)\big\} + |\mathcal{P}_3| \leq \max\big\{\frac{3}{8}a^2, \frac{3}{8}b^2 \big\} + ab \\
&= a\big(\frac{3}{8}a + b\big) \leq  a\big(\frac{1}{2}b + b\big) = \frac{3}{2}ab \\
&= \frac{3}{2}\pi(T),
\end{align*}
as desired.

In what follows, consider $b<\frac{3}{4}a$.
Denote by $r$ the endpoint of $\hat{e}$ with $r\in V(A)$.
Let $B_1(=B),B_2,\ldots,B_d$ be the connected components of $T-r$, where $d=\deg_T(r)$. 
Note that $d\geq 2$ since we assume $n\geq 4$.
For $i=1,2,\ldots,d$ denote by $s_i$ the neighbor of $r$ with $s_i\in V(B_i)$, and let $b_i=|V(B_i)|$.
The assumption $\ell_T(\hat{e})=\pi(T)$ implies that $b_1\geq b_i$ for $i\geq 2$.
Without loss of generality, we assume $b_1\geq b_2\geq \cdots \geq b_d$.
See Fig.~\ref{fig:tree_structure} for the illustration of the structure of $T$.
Notice that $b_1=b\leq a=1+b_2+\cdots+b_d$ and $\pi(T)=ab=b_1(1+b_2+\cdots+b_d)$.

\begin{figure}[h]
\centering
\includegraphics[width=4in]{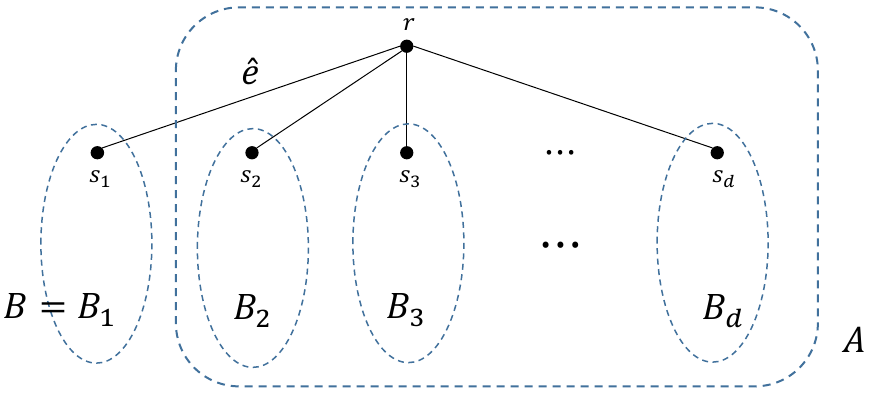}
\caption{The structure of $T$.} \label{fig:tree_structure}
\end{figure}

To render the paper more readable, the rest of the proof is moved to the next two sections.
The framework of the whole proof is illustrated as follows.
\begin{align*}
\begin{cases}
b \geq \frac{3}{4}a & (\text{Section~\ref{sec:main}}) \vspace*{1cm} \\ 
b < \frac{3}{4}a & 
	\begin{cases}
	d = 2 & (\text{Prologue of Section~\ref{sec:d=234}}) \\
	d = 3 & (\text{Section~\ref{sec:d=234-3}}) \\
	d = 4 & (\text{Section~\ref{sec:d=234-4}}) \vspace*{0.5cm} \\
	d \geq 5 & 
		\begin{cases}
		b_1 \geq b_{d-1}+b_d & (\text{Prologue of Section~\ref{sec:d>=5}}) \vspace*{0.3cm} \\
		b_1 < b_{d-1}+b_d & 
			\begin{cases}
			b_1=b_2=\cdots=b_d \ \ (\text{Section~\ref{sec:d>=5-2}})  \\
			b_1=\cdots=b_k > b_{k+1} \geq \cdots \geq b_d \ \ (\text{Section~\ref{sec:d>=5-3}}) 
			\end{cases}
		\end{cases}
	\end{cases}
\end{cases}
\end{align*}
\end{proof}


\begin{remark}\rm
We follow the notation in Fig.~\ref{fig:tree_structure} in this remark.
It is worth noting that the collection of paths crossing the vertex $r$ can be colored properly by $\frac{3}{2}\pi(T)$ colors by simply applying the following result, which was proposed by Shannon~\cite{Shannon49}. 
This result can also be found in \cite[Theorem 7.1.13]{West01}. 
\begin{adjustwidth}{0.8cm}{}\it
If $G$ is a (multi)graph, then the edge-chromatic number of $G$ is upper bounded by $\frac{3}{2}$ times the maximum degree of $G$, i.e., $\chi'(G)\leq\frac{3}{2}\Delta(G)$.
\end{adjustwidth}\rm 
We associate the paths crossing $r$ with a multigraph $\mathcal{H}$ with vertex set $\{v_1,v_2,\ldots,v_d\}$ in such a way that each path whose terminal vertices are in $B_i$ and $B_j$ is represented by an edge connecting $v_i$ and $v_j$.
Notice that the maximum degree of $\mathcal{H}$ is $b(a-1)$, which is less than $\pi(T)$.
By above mentioned result, $\chi'(\mathcal{H})\leq\frac{3}{2}\Delta(\mathcal{H})<\frac{3}{2}\pi(T)$, which leads to a proper path-coloring $\phi$ using at most $\frac{3}{2}\pi(T)$ colors.
Meanwhile, let $T_i$ be the subtree of $T$ induced by $V(B_i)\cup\{r\}$, for $i=1,2,\ldots,d$.
It is obvious that $w(T_i)<\frac{3}{2}\pi(T)$ for each $i$ by induction hypothesis.
As such, the proof of Theorem~\ref{thm:main} can be done if the colors of the paths in each $T_i$ can be appropriately permuted to match the path coloring $\phi$.
However, it is not an easy task to get this well done if there is no further information on the structure of each $T_i$.
This is the reason why our proof is divided into several cases according to the size of each $B_i$. 
\end{remark}

\section{Proof of Theorem~\ref{thm:main} for $b<\frac{3}{4}a$ and $2\leq d\leq 4$} \label{sec:d=234}


We first consider the case when $d=2$.
Let $e$ be the edge connecting $r$ and $s_2$.
If $a>b+1$, one has 
\begin{align*}
\ell(e)=(b+1)(a-1)>ab=\pi(T),
\end{align*}
which contradicts the definition of $\pi(T)$.
Since $b<\frac{3}{4}a$, the inequality $a>b+1$ holds whenever $b\geq 3$.
So, there are only two exceptional cases: $b=1,2$.
With the condition $\frac{4}{3}b<a\leq b+1$, the two exceptional cases are $b=1,a=2$ and $b=2, a=3$, leading to the two paths $P_3$ and $P_5$, respectively.
Note that $P_n$ refers to the path of $n$ vertices.
It is easy to verify that $w(P_3)=3=\pi(P_3)$ and $w(P_5)=6=\pi(P_5)$, so the inequality in \eqref{eq:main} still holds.

The following subsections are devoted to the two cases: $d=3$ and $d=4$.

\subsection{$\mathbf{\textit{\textbf{d}}=3}$} \label{sec:d=234-3}

In this case, $a=1+b_2+b_3$ and $\pi(T)=b_1(1+b_2+b_3)$.
The paths in $R$ are classified into the following 7 classes.
\begin{itemize}
\item $\mathcal{P}_1:=\{P_{r,y}:\,y\in V(T)\setminus\{r\}\}$. 
\item $\mathcal{P}_2:=\{P_{x,y}:\,x\in V(B_1), y\in V(B_2)\}$.
\item $\mathcal{P}_3:=\{P_{x,y}:\,x\in V(B_1), y\in V(B_3)\}$.
\item $\mathcal{P}_4:=\{P_{x,y}:\,x\in V(B_2), y\in V(B_3)\}$. 
\item $\mathcal{P}_5:=\{P_{x,y}:\,x,y\in V(B_1)\}$.
\item $\mathcal{P}_6:=\{P_{x,y}:\,x,y\in V(B_2)\}$.
\item $\mathcal{P}_7:=\{P_{x,y}:\,x,y\in V(B_3)\}$.
\end{itemize}

Firstly, two paths in $\mathcal{P}_1$ can receive the same color if their terminal vertices (except for the one $r$) are not in the same set $V(B_i)$, $i=1,2$ or $3$.
Since all paths in $\{P_{r,y}:\,y\in V(B_i)\}$ share the edge $\{r,s_i\}$ for $i=1,2,3$, one has 
\begin{align*}
\chi(Q(\mathcal{P}_1))\leq\max\{b_1,b_2,b_3\}=b_1.
\end{align*}
In other words, by letting $\mathcal{P}_1^{(i)}=\{P_{r,y}:\,y\in V(B_i)\}$ for $i=1,2,3$, we use the same set of $b_1$ colors to color each $\mathcal{P}_1^{(i)}$, $1\leq i\leq 3$, in such a way that all paths in the same set $\mathcal{P}_1^{(i)}$ receive distinct colors.
Secondly, since any two paths, one in $\mathcal{P}_2$ and the other in $\mathcal{P}_7$, have no edges in common, one has $\chi(Q(\mathcal{P}_2\cup\mathcal{P}_7))\leq\max\{\chi(Q(\mathcal{P}_2))+\chi(Q(\mathcal{P}_7))\}$.
By the fact $\chi(Q(\mathcal{P}_2))=|\mathcal{P}_2|=b_1b_2$, the induction hypothesis that $\chi(Q(\mathcal{P}_7))<\frac{3}{2}\pi(B_3)$ and $\pi(B_3)\leq\frac{1}{4}b_3^2$ from Proposition~\ref{prop:cap-upper}, we have 
\begin{align*}
\chi(Q(\mathcal{P}_2\cup\mathcal{P}_7)) \leq \max\big\{b_1b_2,\frac{3}{8}b_3^2\big\} = b_1b_2.
\end{align*}
Similarly, $\chi(Q(\mathcal{P}_3\cup\mathcal{P}_6)) \leq \max\{b_1b_3,\frac{3}{8}b_2^2\}$ and $\chi(Q(\mathcal{P}_4\cup\mathcal{P}_5)) \leq \max\{b_2b_3,\frac{3}{8}b_1^2\}$.
It follows that
\begin{align}
w(T) &\leq \chi(Q(\mathcal{P}_1)) + \chi(Q(\mathcal{P}_2\cup\mathcal{P}_7)) + \chi(Q(\mathcal{P}_3\cup\mathcal{P}_6)) + \chi(Q(\mathcal{P}_4\cup\mathcal{P}_5)) \label{eq:2-upper-def} \\
&\leq b_1 + b_1b_2 + \max\big\{b_1b_3,\frac{3}{8}b_2^2\big\} + \max\big\{b_2b_3,\frac{3}{8}b_1^2\big\}. \label{eq:2-upper}
\end{align}

Consider the case when $b_2b_3\geq\frac{3}{8}b_1^2$.
Since $b_1b_3\geq\frac{3}{8}b_2^2$ due to $b_1\geq b_2\geq b_3$, it follows from \eqref{eq:2-upper} that
\begin{align*}
w(T) &\leq b_1 + b_1b_2 + b_1b_3 + b_2b_3 = b_1(1+b_2+b_3) + \frac{1}{2}b_2b_3 + \frac{1}{2}b_2b_3 \\
&\leq b_1(1+b_2+b_3) + \frac{1}{2}b_1b_2 + \frac{1}{2}b_1b_3 <  b_1(1+b_2+b_3) + \frac{1}{2}b_1(1+b_2+b_3) \\
&= \frac{3}{2}\pi(T).
\end{align*}

Next, consider the case when $b_2b_3<\frac{3}{8}b_1^2$.
If $b_1b_3\geq\frac{3}{8}b_2^2$, by \eqref{eq:2-upper} and the assumption that $b<\frac{3}{4}a$ we have
\begin{align*}
w(T) &\leq b_1 + b_1b_2 + b_1b_3 + \frac{3}{8}b_1^2 < b_1(1+b_2+b_3) + \frac{9}{32}b_1(1+b_2+b_3) \\
&< \frac{3}{2}b_1(1+b_2+b_3) = \frac{3}{2}\pi(T).
\end{align*}
Otherwise, by \eqref{eq:2-upper} and $b<\frac{3}{4}a$ again we have
\begin{align}
w(T) &\leq b_1 + b_1b_2 + \frac{3}{8}b_2^2 + \frac{3}{8}b_1^2 \leq b_1 + b_1b_2 + \frac{3}{4}b_1^2 \notag \\ 
& < b_1 + b_1b_2 + \frac{9}{16}b_1(1+b_2+b_3) = \frac{3}{2}b_1(1+b_2+b_3) + \frac{1}{16}b_1(1+b_2-15b_3). \label{eq:2-upper-2}
\end{align}
Notice that $b<\frac{3}{4}a$ and $b_2\leq b_1$ imply $b_2\leq 3+3b_3$.
It follows from \eqref{eq:2-upper-2} that
\begin{align*}
w(T) < \frac{3}{2}b_1(1+b_2+b_3) - \frac{1}{4}b_1(3b_3-1) < \frac{3}{2}b_1(1+b_2+b_3) = \frac{3}{2}\pi(T).
\end{align*}

\subsection{$\mathbf{\textit{\textbf{d}}=4}$} \label{sec:d=234-4}

In this case, $a=1+b_2+b_3+b_4$ and $\pi(T)=b_1(1+b_2+b_3+b_4)$.
We consider the following sub-cases.
\begin{enumerate}[(i)]
\item $b_1\geq b_3+b_4$.
\item $b_1<b_3+b_4$ and $b_1b_4\geq b_2b_3$.
\item $b_1<b_3+b_4$, $b_1b_4<b_2b_3$ and $b_1\leq 4b_4$.
\item $b_1<b_3+b_4$, $b_1b_4<b_2b_3$ and $b_1> 4b_4$.
\end{enumerate}

\subsubsection{\textbf{Proof for sub-case (i)}}

We first obtain a new graph $T'$ from $T$ by removing the edge $\{r,s_4\}$ and adding the edge $\{s_3,s_4\}$, see Fig.~\ref{fig:T'} for an example of $T'$.
It is not hard to see that the edge $\hat{e}$ still maximizes the value $\ell_{T'}(e)$, which implies that $\pi(T')=b_1(1+b_2+b_3+b_4)=\pi(T)$.
Observe that $\deg_{T'}(r)=3$.
By the same argument in Section~\ref{sec:d=234-3} for $d=3$ case, we have
\begin{equation}
w(T') < \frac{3}{2}\pi(T') = \frac{3}{2}\pi(T).
\label{eq:T'-exchange-upper}
\end{equation}

\begin{figure}[h]
\centering
\includegraphics[width=5in]{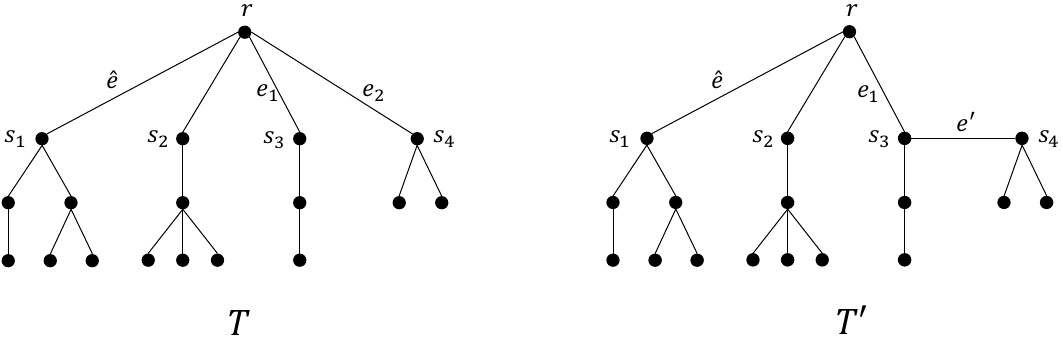}
\caption{$T'=T-\{r,s_4\}+\{s_3,s_4\}$.} \label{fig:T'}
\end{figure}

Denote by $B'_1,B'_2,B'_3$ the three connected components of $T'-r$ with $B'_1=B_1$ and $B'_2=B_2$.
Let $\phi'$ be the proper path-coloring of $T'$ given in the argument in Secion~\ref{sec:d=234-3}.
More precisely, divide all paths in $T'$ into 7 sets
\begin{align*}
\mathcal{P}'_1 &:=\{P'_{r,y}:\,y\in V(T')\setminus\{r\}\} \\
\mathcal{P}'_2 &:=\{P'_{x,y}:\,x\in V(B'_1), y\in V(B'_2)\} \\
\mathcal{P}'_3 &:=\{P'_{x,y}:\,x\in V(B'_1), y\in V(B'_3)\} \\
\mathcal{P}'_4 &:=\{P'_{x,y}:\,x\in V(B'_2), y\in V(B'_3)\} \\
\mathcal{P}'_5 &:=\{P'_{x,y}:\,x,y\in V(B'_1)\} \\
\mathcal{P}'_6 &:=\{P'_{x,y}:\,x,y\in V(B'_2)\} \\
\mathcal{P}'_7 &:=\{P'_{x,y}:\,x,y\in V(B'_3)\}
\end{align*}
and color the paths in $\mathcal{P}'_1$ with $\max\{|B'_1|,|B'_2|,|B'_3|\}=b_1$ colors, the paths in $\mathcal{P}'_2\cup\mathcal{P}'_7$ with $\chi(Q(\mathcal{P}'_2\cup\mathcal{P}'_7))$ colors, the paths in $\mathcal{P}'_3\cup\mathcal{P}'_6$ with $\chi(Q(\mathcal{P}'_3\cup\mathcal{P}'_6))$ colors, and the paths in $\mathcal{P}'_4\cup\mathcal{P}'_5$ with $\chi(Q(\mathcal{P}'_4\cup\mathcal{P}'_5))$ colors.
We use the superscript ``prime'' herein to emphasize the paths are considered in $T'$.
Note that, by the arguments in~\eqref{eq:2-upper-def}--\eqref{eq:2-upper}, the four classes of colors respectively used in $\mathcal{P}'_1$, $\mathcal{P}'_2\cup\mathcal{P}'_7$, $\mathcal{P}'_3\cup\mathcal{P}'_6$ and $\mathcal{P}'_4\cup\mathcal{P}'_5$ are mutually disjoint.
Note also that, for $1\leq i\leq 3$, all paths in $\{P'_{r,y}:\,y\in V(B'_i)\}\subset\mathcal{P}'_1$ receive distinct colors since they have the common edge $\{r,s_i\}$.

Define a path-coloring $\phi$ of $T$ by
\begin{subnumcases}{\phi(P_{x,y}) =}
\phi'(P'_{s_3,y}), & if $x=r$ and $y\in V(B_4)$; \label{eq:phi-1} \\
\phi'(P'_{r,y}), & if $x=s_3$ and $y\in V(B_4)$; \label{eq:phi-2} \\
\phi'(P'_{x,y}), & otherwise. \label{eq:phi-3}
\end{subnumcases}
$\phi$ is well-defined since it just exchanges the color of the path connecting $r$ and $y$ with that connecting $s_3$ and $y$, for any $y\in V(B_4)$.

\begin{lemma}\label{lemma:phi-proper}
The path-coloring defined on $T$ in \eqref{eq:phi-1}--\eqref{eq:phi-3} is proper.
\end{lemma}
\begin{proof}
Following the definition of $\phi$, we first define a mapping $f$ from $R$ to the routing of $T'$ as
\begin{align*}
f(P_{x,y}) = 
\begin{cases}
P'_{s_3,y}, & \text{if } x=r \text{ and } y\in V(B_4); \\
P'_{r,y}, & \text{if } x=s_3 \text{ and } y\in V(B_4); \\
P'_{x,y}, & \text{otherwise.}
\end{cases}
\end{align*}
As such, $\phi(P_{x,y})=\phi'(f(P_{x,y}))$.
By the structures of $T$ and $T'$, we have $f(P_{x,y})\neq P_{x,y}$ if and only if exactly one of $x,y$ is in $V(B_4)$, i.e., one endpoint is in $V(B_4)$ and the other is in $V(B_1)\cup V(B_2)\cup V(B_3)\cup\{r\}$.
For convenience, use $e_1,e_2$ and $e'$ to denote the edges $\{r,s_3\}$, $\{r,s_4\}$ and $\{s_3,s_4\}$, respectively, as shown in Fig.~\ref{fig:T'}.

Consider two paths in $T$, say $P_{x,y}$ and $P_{u,v}$, which have at least one common edge.
Let $\mathcal{E}$ (resp., $\mathcal{E}'$) denote the collection of common edges of $P_{x,y}$ and $P_{u,v}$ (resp., $f(P_{x,y})$ and $f(P_{u,v})$).
Notice that 
\begin{equation}\label{eq:proof-subcase-1}
\phi(P_{x,y})=\phi'(f(P_{x,y}))\neq\phi'(f(P_{u,v}))=\phi(P_{u,v}) \text{ whenever } \mathcal{E}'\neq\emptyset.
\end{equation}

When $f(P_{x,y})=P_{x,y}$ and $f(P_{u,v})=P_{u,v}$, one has $\mathcal{E}'=\mathcal{E}\neq\emptyset$, and thus the result follows by~\eqref{eq:proof-subcase-1}.
Consider the case when $f(P_{x,y})\neq P_{x,y}$ and $f(P_{u,v})\neq P_{u,v}$.
By assuming $y,v\in V(B_4)$, we have $x,u\in V(B_1)\cup V(B_2)\cup V(B_3)\cup\{r\}$.
For either case, it is no hard to check that $e'\in\mathcal{E}'$.
Hence $\mathcal{E}'\neq\emptyset$, and the result follows by~\eqref{eq:proof-subcase-1}.

Finally, we consider by symmetry that $f(P_{x,y})\neq P_{x,y}$ and $f(P_{u,v})=P_{u,v}$.
By assuming $y\in V(B_4)$, one has $x\in V(B_1)\cup V(B_2)\cup V(B_3)\cup\{r\}$.
When $x\in V(B_1)\cup V(B_2)$, the assumptions $\mathcal{E}\neq\emptyset$ and $f(P_{u,v})=P_{u,v}$ imply three possibilities of the locations of $u$ and $v$: $u,v\in V(B_1)\cup V(B_2)$, $u,v\in V(B_4)$, or one of them is in $V(B_1)\cup V(B_2)$ and the other is in $V(B_3)\cup\{r\}$.
For either case, one has $\mathcal{E}'=\mathcal{E}\neq\emptyset$. 
When $x=r$, the assumptions $\mathcal{E}\neq\emptyset$ and $f(P_{u,v})=P_{u,v}$ imply that $u,v\in V(B_4)$.
Then, we get $\mathcal{E}'=\mathcal{E}\neq\emptyset$.
Finally, consider the case when $x\in V(B_3)$.
By $\mathcal{E}\neq\emptyset$ and $f(P_{u,v})=P_{u,v}$ again, we have $u,v\in V(B_4)$ or $u,v\in V(B_3)\cup\{r\}$.
In the former case we have $\mathcal{E}'=\mathcal{E}\neq\emptyset$, so it suffices to consider the latter case.
If $e_1\notin\mathcal{E}$, the common edges in $\mathcal{E}$ must be in $B_3$, and thus $\mathcal{E}'=\mathcal{E}\neq\emptyset$.
Otherwise, if $e_1\in\mathcal{E}$, this indicates that $r$ is an endpoint of $P_{u,v}$, which implies that $f(P_{u,v})\in\mathcal{P}'_1$.
Indeed, $f(P_{u,v})=P'_{r,v}$ and $v\in V(B'_3)$, by assuming $u=r$.
There are two sub-cases: $x=s_3$ or $x\in V(B_3)\setminus\{s_3\}$.
If $x=s_3$, then $f(P_{x,y})=P'_{r,y}$, where $y\in V(B'_3)$.
We get $\mathcal{E}'\neq\emptyset$ since $\{r,s_3\}$ is a common edge of $f(P_{x,y})$ and $f(P_{u,v})$.
If $x\in V(B_3)\setminus\{s_3\}$, since $y\in V(B_4)$, one has $x,y\in V(B'_3)$ and then $f(P_{x,y})\in\mathcal{P}'_7$.
By the arguments in~\eqref{eq:2-upper-def}--\eqref{eq:2-upper}, the definition of $\phi'$ assuming the colors in $\mathcal{P}'_1$ are different from those in $\mathcal{P}'_7$, hence we have 
\begin{align*}
\phi(P_{x,y})=\phi'(f(P_{x,y}))\neq\phi'(f(P_{u,v}))=\phi(P_{u,v}).
\end{align*}
This completes the proof.
\end{proof}

We remark here that one key point in the proof of Lemma~\ref{lemma:phi-proper} is the assumption that in $\phi'$, the color on any path with one endpoint $r$ and the other in $V(B'_3)$ is different from the color on any path with two endpoints in $V(B'_3)$.

Lemma~\ref{lemma:phi-proper} guarantees that $w(T)\leq w(T')$.
Hence the result follows by \eqref{eq:T'-exchange-upper}.

\subsubsection{\textbf{Proof for sub-cases (ii) and (iii)}} \label{sec:d=234-4-23}

We consider (ii) and (iii) simultaneously.
The paths in $R$ are classified into the following 5 classes.
\begin{itemize}
\item $\mathcal{P}_1:=\{P_{r,y}:\, y\in V(T)\setminus\{r\}\}$. 
\item $\mathcal{P}_2:=\{P_{x,y}:\, x,y\in V(B_i) \text{ for some } i, 1\leq i\leq 4\}$.
\item $\mathcal{P}_3:=\{P_{x,y}:\, x\in V(B_1) \text{ and } y\in V(B_2) \text { or } x\in V(B_3) \text{ and } y\in V(B_4)\}$.
\item $\mathcal{P}_4:=\{P_{x,y}:\, x\in V(B_1) \text{ and } y\in V(B_3) \text { or } x\in V(B_2) \text{ and } y\in V(B_4)\}$.
\item $\mathcal{P}_5:=\{P_{x,y}:\, x\in V(B_1) \text{ and } y\in V(B_4) \text { or } x\in V(B_2) \text{ and } y\in V(B_3)\}$.
\end{itemize}
Similar to the argument in \eqref{eq:2-upper} of Section~\ref{sec:d=234-3} for $d=3$ case, we have
\begin{align*}
w(T) \leq~~ & b_1 + \max\big\{\frac{3}{8}b_1^2,\frac{3}{8}b_2^2,\frac{3}{8}b_3^2,\frac{3}{8}b_4^2\big\} + \max\{b_1b_2,b_3b_4\} \\
&+ \max\{b_1b_3,b_2b_4\} + \max\{b_1b_4,b_2b_3\}. 
\end{align*}
By the assumption that $b_1\geq b_2\geq b_3\geq b_4$, the above inequality can be simplified as
\begin{equation}
w(T) \leq b_1 + \frac{3}{8}b_1^2 + b_1b_2 + b_1b_3 + \max\{b_1b_4,b_2b_3\}.
\label{eq:4-upper}
\end{equation}

Consider (ii): $b_1<b_3+b_4$ and $b_1b_4\geq b_2b_3$. 
It follows from \eqref{eq:4-upper} and $b<\frac{3}{4}a$ that
\begin{align*}
w(T) &\leq b_1 + \frac{3}{8}b_1^2 + b_1b_2 + b_1b_3 + b_1b_4 \\ 
&< b_1(1+b_2+b_3+b_4) + \frac{3}{8}b_1\cdot\frac{3}{4}(1+b_2+b_3+b_4) \\ 
&=\frac{41}{32}b_1(1+b_2+b_3+b_4) < \frac{3}{2}b_1(1+b_2+b_3+b_4) \\
&= \frac{3}{2}\pi(T).
\end{align*}

Consider (iii): $b_1<b_3+b_4$, $b_1b_4<b_2b_3$ and $b_1\leq 4b_4$.
It follows from \eqref{eq:4-upper} that
\begin{align}
w(T) &\leq b_1 + \frac{3}{8}b_1^2 + b_1b_2 + b_1b_3 + b_2b_3 \notag \\ 
&= \frac{3}{2}b_1(1+b_2+b_3+b_4) + \frac{3}{8}b_1^2 + b_2b_3 - \frac{1}{2}b_1 - \frac{1}{2}b_1b_2 - \frac{1}{2}b_1b_3 - \frac{3}{2}b_1b_4 \label{eq:(iii)-1} \\ 
&\leq \frac{3}{2}b_1(1+b_2+b_3+b_4) + \frac{3}{8}b_1(b_1 - 4b_4 - \frac{4}{3}) \label{eq:(iii)-2} \\
&< \frac{3}{2}b_1(1+b_2+b_3+b_4) \label{eq:(iii)-3} \\
&= \frac{3}{2}\pi(T), \notag
\end{align}
where inequalities \eqref{eq:(iii)-2} and \eqref{eq:(iii)-3} are due to $2b_2b_3\leq b_1 b_2+ b_1b_3$ and $b_1\leq 4b_4$, respectively.

\subsubsection{\textbf{Proof for sub-case (iv)}} \label{sec:d=234-4-4}

Finally, consider (iv): $b_1<b_3+b_4$, $b_1b_4<b_2b_3$ and $b_1> 4b_4$.
The paths in $R$ are classified into the following 7 classes.
\begin{itemize}
\item $\mathcal{P}_1:=\{P_{r,y}:\, y\in V(T)\setminus\{r\}\}$. 
\item $\mathcal{P}_2:=\{P_{x,y}:\, x\in V(B_1) \text{ and } y\in V(B_2)\}$.
\item $\mathcal{P}_3:=\{P_{x,y}:\, x,y\in V(B_3)\cup V(B_4)\}$.
\item $\mathcal{P}_4:=\{P_{x,y}:\, x\in V(B_1) \text{ and } y\in V(B_3)\}$.
\item $\mathcal{P}_5:=\{P_{x,y}:\, x,y\in V(B_2)\cup V(B_4)\}$.
\item $\mathcal{P}_6:=\{P_{x,y}:\, x\in V(B_2) \text{ and } y\in V(B_3)\}$.
\item $\mathcal{P}_7:=\{P_{x,y}:\, x,y\in V(B_1)\cup V(B_4)\}$.
\end{itemize}

Since $b_1\geq b_2\geq b_3\geq b_4$ and each path $P_{r,y}$ with $y\in V(B_1)$ contains the edge $\{r,s_1\}$, we have $\chi(Q(\mathcal{P}_1))=b_1$.
As each path in $\mathcal{P}_2$ contains the two edges $\{r,s_1\},\{r,s_2\}$, it is clear that $\chi(Q(\mathcal{P}_2))=b_1b_2$.
Similarly, we have $\chi(Q(\mathcal{P}_4))=b_1b_3$, and $\chi(Q(\mathcal{P}_6))=b_2b_3$.
Let $\widehat{B}$ be the tree obtained from the union of $B_3$ and $B_4$ by adding an extra edge connecting $s_3$ and $s_4$.
It is easy to see that $w(\widehat{B})=\chi(Q(\mathcal{P}_3))$. 
By the induction hypothesis that $w(\widehat{B})<\frac{3}{2}\pi(\widehat{B})$, it follows from Proposition~\ref{prop:cap-upper} that 
\begin{align*}
\chi(Q(\mathcal{P}_3)) = w(\widehat{B}) <\frac{3}{2}\pi(\widehat{B}) \leq \frac{3}{8}(b_3+b_4)^2.
\end{align*}
Furthermore, since any two paths, one in $\mathcal{P}_2$ and another in $\mathcal{P}_3$, can receive the same color, we have 
\begin{equation}
\chi(Q(\mathcal{P}_2\cup\mathcal{P}_3)) = \max\big\{\chi(Q(\mathcal{P}_2)),\chi(Q(\mathcal{P}_3))\big\} \leq \max\big\{b_1b_2,\frac{3}{8}(b_3+b_4)^2\big\}.
\label{eq:d=4-4-P2P3}
\end{equation}
By the same argument, one has
\begin{equation}
\chi(Q(\mathcal{P}_4\cup\mathcal{P}_5)) \leq \max\big\{b_1b_3,\frac{3}{8}(b_2+b_4)^2\big\}
\label{eq:d=4-4-P4P5}
\end{equation}
and
\begin{equation}
\chi(Q(\mathcal{P}_6\cup\mathcal{P}_7)) \leq \max\big\{b_2b_3,\frac{3}{8}(b_1+b_4)^2\big\}.
\label{eq:d=4-4-P6P7}
\end{equation}
Combining $\chi(Q(\mathcal{P}_1))=b_1$ and equations \eqref{eq:d=4-4-P2P3}--\eqref{eq:d=4-4-P6P7} yields
\begin{align}
w(T) \leq & \chi(Q(\mathcal{P}_1)) + \sum_{i=1}^3 \chi\left(Q\left(\mathcal{P}_{2i}\cup\mathcal{P}_{2i+1}\right)\right) \notag \\
\leq & b_1 + \max\big\{b_1b_2,\frac{3}{8}(b_3+b_4)^2\big\} + \max\big\{b_1b_3,\frac{3}{8}(b_2+b_4)^2\big\} \notag \\
& \, \ \ + \max\big\{b_2b_3,\frac{3}{8}(b_1+b_4)^2\big\}. \label{eq:4-4-upper-1}
\end{align}

As $b_1<b_3+b_4$ and $b_1>4b_4$ implying $b_4<\frac{1}{3}b_3$ and $b_1<\frac{4}{3}b_3$, by $b_4\leq b_3\leq b_2\leq b_1$ we have
\begin{equation}
\frac{3}{8}(b_3+b_4)^2 < \frac{3}{8}(b_1+\frac{1}{4}b_1)(b_2+b_2) = \frac{15}{16}b_1b_2,
\label{eq:d=4-4-P3-chi-1}
\end{equation}
and
\begin{align}
\frac{3}{8}(b_2+b_4)^2 &< \frac{3}{8}(b_1+\frac{1}{4}b_1)^2 = \frac{75}{128}b_1^2 \notag \\
&< \frac{75}{128}b_1\cdot\frac{4}{3}b_3 = \frac{25}{32}b_1b_3. \label{eq:d=4-4-P5-chi-1}
\end{align}
By plugging \eqref{eq:d=4-4-P3-chi-1} and\eqref{eq:d=4-4-P5-chi-1} into \eqref{eq:4-4-upper-1}, we have
\begin{equation}
w(T) \leq b_1 + b_1b_2 + b_1b_3 + \max\big\{b_2b_3,\frac{3}{8}(b_1+b_4)^2\big\}.
\label{eq:4-4-upper-2}
\end{equation}

If $b_2b_3\geq\frac{3}{8}(b_1+b_4)^2$, it follows \eqref{eq:4-4-upper-2} that
\begin{align*}
w(T) & \leq b_1 + b_1b_2 + b_1b_3 + b_2b_3 \\
& = \frac{3}{2}b_1\big( 1 + b_2 + b_3 + b_4 \big) + \big( b_2b_3 - \frac{1}{2}b_1(b_2+b_3) \big)- \frac{1}{2}b_1 - \frac{3}{2} b_1b_4 \\
& \leq \frac{3}{2}\pi(T) - \frac{1}{2}b_1 - \frac{3}{2} b_1b_4 \\
& < \frac{3}{2}\pi(T).
\end{align*} 

Else, if $b_2b_3<\frac{3}{8}(b_1+b_4)^2$, it follows \eqref{eq:4-4-upper-2} again that 
\begin{align}
w(T) &\leq b_1 + b_1b_2 + b_1b_3 + \frac{3}{8}(b_1+b_4)^2 \notag \\
&= \frac{3}{2}b_1(1+b_2+b_3+b_4) - \frac{1}{2}b_1 - \frac{1}{2}b_1b_2 - \frac{1}{2}b_1b_3 - \frac{3}{4}b_1b_4 + \frac{3}{8}b_1^2 + \frac{3}{8}b_4^2 \notag \\
&= \frac{3}{2}\pi(T) - \frac{1}{2}b_1 - \frac{1}{2}b_1b_2 + \frac{3}{8}b_1\big(b_1-\frac{4}{3}b_3\big) + \frac{3}{8}b_4\big(b_4-2b_1\big) \notag \\
&< \frac{3}{2}\pi(T) - \frac{1}{2}b_1 - \frac{1}{2}b_1b_2 \label{eq:4-4-upper} \\
&< \frac{3}{2}\pi(T), \notag
\end{align}
where \eqref{eq:4-4-upper} is due to $b_1<\frac{4}{3}b_3$ and $b_4\leq b_1$.

\section{Proof of Theorem~\ref{thm:main} for $b<\frac{3}{4}a$ and $d\geq 5$} \label{sec:d>=5}

We consider three scenarios: (i) $b_1\geq b_{d-1}+b_{d}$; (ii) $b_1< b_{d-1}+b_{d}$ and $b_1=b_2=\cdots=b_d$; and (iii) $b_1< b_{d-1}+b_{d}$ and $b_1=b_2=\cdots=b_k>b_{k+1}\geq\cdots\geq b_d$ for some $k$ with $1\leq k\leq d-1$.

The first case can be dealt with by a similar argument used in the proof of Lemma~\ref{lemma:phi-proper}.
More precisely, since $b<\frac{3}{4}a$ and $d\geq 5$, there exists $3\leq t\leq d-1$ such that  
\begin{equation}\label{eq:generalLemma3-t}
b_t+b_{t+1}+\cdots+b_d \leq b_1 < b_{t-1}+b_t+b_{t+1}+\cdots+b_d.
\end{equation}
We can obtain a new tree $T'$ from $T$ by removing the edges $\{r,s_{t+1}\},\{r,s_{t+2}\},\ldots,\{r,s_d\}$ and adding the edges $\{s_t,s_{t+1}\},\{s_t,s_{t+2}\},\ldots,\{s_t,s_d\}$, see Fig.~\ref{fig:T'-general} for an example of $T'$.

\begin{figure}[h]
\centering
\includegraphics[width=4.7in]{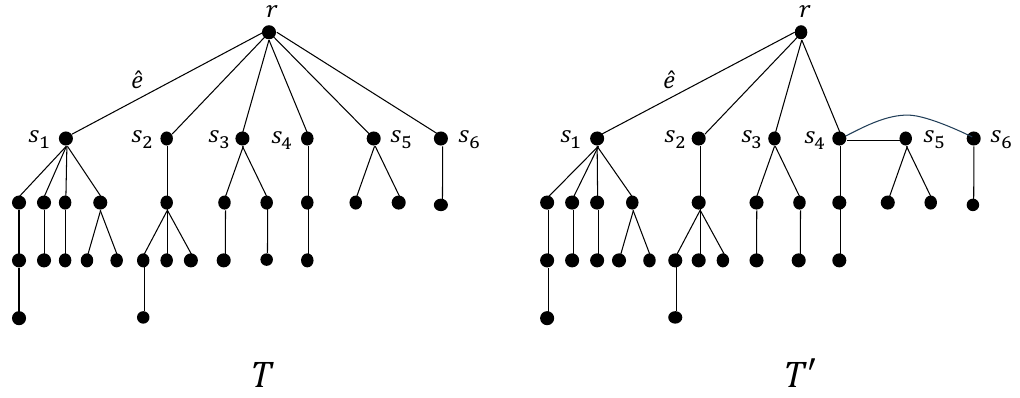}
\caption{$t=4:\ T'=T-\{r,s_5\}-\{r,s_6\}+\{s_4,s_5\}+\{s_4,s_6\}$.} \label{fig:T'-general}
\end{figure}

It is not hard to see that $\hat{e}$ still maximizes the value $\ell_{T'}(e)$ and $\pi(T')=\pi(T)$.
Observe that $T'$ is a tree rooted by $r$ with $t\geq 3$ branches.
Let $B'_1=B_1,\ldots,B'_{t-1}=B_{t-1}$ and $B'_{t}$ be the $t$ connected components of $T'-r$, and denoted by $b'_i=|V(B'_i)|$ for $i=1,2,\ldots,t$.
It follows from~\eqref{eq:generalLemma3-t} that $b'_1<b'_{t-1}+b'_t$.
Let $\phi'$ be the proper path-coloring of $T'$ given in Section~\ref{sec:d=234-3}, Section~\ref{sec:d=234-4-23}, Section~\ref{sec:d=234-4-4} or the one that will be proposed later in Section~\ref{sec:d>=5-2} or Section~\ref{sec:d>=5-3}, according to the number of branches, $t$.
We have 
\begin{equation}\label{eq:generalLemma3-proof-1}
w(T')<\frac{3}{2}\pi(T')=\frac{3}{2}\pi(T).
\end{equation}
In either case, the proposed path-coloring $\phi'$ has the property that 
\begin{equation}\label{eq:generalLemma3-proof-2}
\{\phi'(P_{x,y}):\,x,y\in V(B'_t)\} \cap \{\phi'(P_{r,v}):\,v\in V(B'_t)\} = \emptyset.
\end{equation}
Moreover, $|\{\phi'(P_{r,v}):v\in V(B'_t)\}|=b'_t$, namely, the colors on the paths in this set are all distinct, since $\{r,s_t\}$ is a common edge of all these paths.
Now, define a path-coloring $\phi$ of $T$ by
\begin{align*}
\phi(P_{x,y}) = \begin{cases}
	\phi'(P'_{s_t,y}), & \text{if } x=r \text{ and } y\in V(B_{t+1})\cup\cdots\cup V(B_{d}); \\
 	\phi'(P'_{r,y}), & \text{if } x=s_t \text{ and } y\in V(B_{t+1})\cup\cdots\cup V(B_{d});  \\
	\phi'(P'_{x,y}), & \text{otherwise}.
\end{cases}
\end{align*}
By the similar argument in the proof of Lemma~\ref{lemma:phi-proper}, it is routine to verify that $\phi$ is proper.
Hence we have $w(T)<\frac{3}{2}\pi(T)$ by \eqref{eq:generalLemma3-proof-1}.


The following subsections are devoted to the rest two cases.
Note that, since $d\geq 5$, the condition $b_1<b_{d-1}+b_d$ implies that $b<a/2$.



\subsection{$\boldsymbol{b_1< b_{d-1}+b_{d}}$ \textbf{and} $\boldsymbol{b_1=b_2=\cdots=b_d}$} \label{sec:d>=5-2}
In this case, $\pi(T)=b(bd-b+1)$.
Consider the following classification of $R$.
\begin{itemize}
\item $\mathcal{P}_i:=\big\{P_{x,y}:\, x,y\in V(B_i)\cup\{r\}\big\}$, for $i=1,2,\ldots,d$; and
\item $\mathcal{P}_{(i,j)}:=\{P_{x,y}:\,x\in V(B_i) \text{ and } y\in V(B_j)\}$, for $1\leq i<j \leq d$.
\end{itemize}

First, assume $d$ is even. 
For $i\neq i'$, any two paths, one in $\mathcal{P}_i$ and another in $\mathcal{P}_{i'}$, can receive the same color.
Then, by the induction hypothesis that $w(B_1+r)<\frac{3}{2}\pi(B_1+r)$ and Proposition~\ref{prop:cap-upper} we have
\begin{align*}
\chi\left(Q\left(\bigcup_{i=1}^d\mathcal{P}_i\right)\right) = \chi(Q(\mathcal{P}_1)) = w(B_1+r) < \frac{3}{2}\pi(B_1+r) \leq \frac{3}{8}(b+1)^2,
\end{align*}
$B_1+r$ refers to the induced subgraph of $T$ by the vertex set $V(B_1)\cup\{r\}$.

Recall that the chromatic index of a complete graph of order $d$ is $d-1$ when $d$ is even.
Let $K_d$ be a complete graph of $d$ vertices labelled $1,2,\ldots,d$, and let $f:E(K_d)\to\{1,2,\ldots,d-1\}$ be a proper $(d-1)$-edge-coloring of $K_d$.
For $t=1,2,\ldots,d-1$, denote by $C_t$ the collection of ordered pairs $(i,j)$ such that $f(\{i,j\})=t$, where $i<j$.
Any two paths, one in $\mathcal{P}_{(i,j)}$ and another in $\mathcal{P}_{(i',j')}$, can receive the same color if $(i,j)$ and $(i',j')$ are distinct and both in $C_t$ for some $t$.
This implies that, for any $t$,
\begin{align*}
\chi\left(Q\left(\bigcup_{(i,j)\in C_t}\mathcal{P}_{(i,j)}\right)\right) = \chi\left(Q(\mathcal{P}_{(i,j)})\right) = |\mathcal{P}_{(i,j)}| = b^2.
\end{align*}
To sum up, we have
\begin{align*}
w(T) &\leq \chi\left(Q\left(\bigcup_{i=1}^d\mathcal{P}_i\right)\right) + \sum_{t=1}^{d-1} \chi\left(Q\left(\bigcup_{(i,j)\in C_t}\mathcal{P}_{(i,j)}\right)\right) \\
&< \frac{3}{8}(b+1)^2 + (d-1)b^2 < \frac{3}{2}b(bd-b+1) = \frac{3}{2}\pi(T).
\end{align*}

Second, assume $d$ is odd.
Recall that the total-chromatic number of a graph $G$ is the minimum integer $k$ needed to guarantee the existence of a mapping from $V(G)\cup E(G)$ to a set of $k$ colors such that (i) adjacent vertices receive distinct colors, (ii) adjacent edges receive distinct colors, and (iii) any vertex and its incident edges receive distinct colors.
The total-chromatic number of $K_d$ is known to be $d$ when $d$ is odd, see~\cite[p.16]{Yap96}.

For convenience, label the set of vertices in $K_d$ by $1,2,\ldots,d$.
Let $f:V(K_d)\cup E(K_d)\to\{1,2,\ldots,d\}$ be a proper $d$-total-coloring of $K_d$ such that $f(t)=t$ for any $t\in V(K_d)$.
For $i=1,2,\ldots,d$, we further decompose $\mathcal{P}_i$ into two classes of paths 
\begin{align*}
\mathcal{P}^{(r)}_i:=\{P_{r,y}:\,y\in V(B_i)\} \quad \text{and} \quad \mathcal{P}^{(B)}_i:=\{P_{x,y}:\,x,y\in V(B_i)\}.
\end{align*}
Observe that $\chi(Q(\mathcal{P}^{(r)}_i))=b$.
By a similar argument, for $t=1,2,\ldots,d$, one has
\begin{align*}
\chi\left(Q\left(\mathcal{P}_t\cup\bigcup_{f(\{i,j\})=t}\mathcal{P}_{(i,j)}\right)\right) &= \chi\left(Q\left(\Big(\mathcal{P}^{(r)}_t\cup\mathcal{P}^{(B)}_t\Big)\cup\bigcup_{f(\{i,j\})=t}\mathcal{P}_{(i,j)}\right)\right) \\ 
&\leq \max\Big\{b+\frac{3}{8}b^2, b^2\Big\},
\end{align*}
which is equal to $b^2$ whenever $b\geq 2$.
Therefore, for $b\geq 2$,
\begin{align*}
w(T) \leq \sum_{t=1}^d \chi\left(Q\left(\mathcal{P}_t\cup\bigcup_{f(\{i,j\})=t}\mathcal{P}_{(i,j)}\right)\right) \leq d\cdot b^2 < \frac{3}{2}b(bd-b+1) = \frac{3}{2}\pi(T).
\end{align*}
For the exceptional case that $b=1$, the tree $T$ is a star with $d$ leaves.
By the proper $d$-total-coloring method, we have $w(T)=d$.
Then, the assertion in \eqref{eq:main} holds since $\pi(T)=d$.

It is worth mentioning that in the case when $d$ is odd, \cite[Lemma 5]{FLWZ17} claimed that $w(T)=d\cdot b^2$.
The result $w(T)<\frac{3}{2}\pi(T)$ just follows from $\pi(T)=b(bd-b+1)$ and the assumption $d\geq 5$.
However, to make sure that the proposed path-coloring satisfies the condition in~\eqref{eq:generalLemma3-proof-2}, we still have to write down the concrete path-coloring as shown above.

\medskip

\subsection{$\boldsymbol{b_1< b_{d-1}+b_{d}}$ \textbf{and} $\boldsymbol{b_1=b_2=\cdots=b_k>b_{k+1}}$ \textbf{for some} $\boldsymbol{k}$ \textbf{with} $\boldsymbol{1\leq k\leq d-1}$} \label{sec:d>=5-3}

Construct a tree $T'$ from $T$ by removing a leave $u_i$ from $B_i$, where $u_i$ is arbitrarily chosen, for $i=1,2,\ldots,k$.
Let $B_i'=B_i-u_i$ for $i=1,2,\ldots,k$.
See Fig.~\ref{fig:T'-delete} for an example of $T$ and $T'$ with $d=5$ and $k=3$.

\begin{figure}[h]
\centering
\includegraphics[width=5.5in]{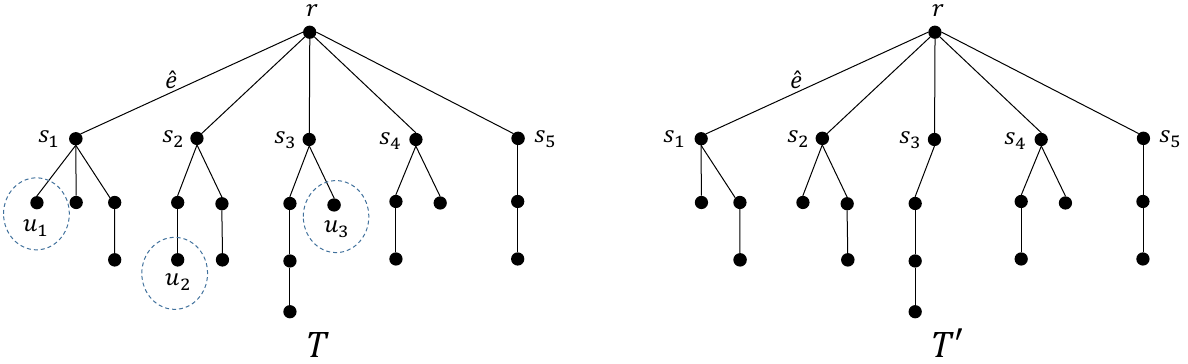}
\caption{$T'=T-u_1-u_2-u_3$, where $u_i$ is an arbitrarily chosen leaf for $i=1,2,3$.} \label{fig:T'-delete}
\end{figure}

Notice that $T'-r$ has $d$ branches $B'_1,\ldots,B'_k,B_{k+1},\ldots,B_d$.
We now claim that $\hat{e}$ still maximizes the value $\ell_{T'}(e)$ among all edges in $T'$, that is, $\ell_{T'}(\hat{e})=\pi(T')$.
Suppose to the contrary that $\ell_{T'}(\hat{e}')>\ell_{T'}(\hat{e})$ for some $\hat{e}'\neq\hat{e}$ in $T'$.
Observe that $\hat{e}'$ can not be in any connected component in $T'-r$, otherwise, we will get $\ell_{T}(\hat{e}')>\ell_{T}(\hat{e})$ due to the structures of $T$ and $T'$, which is a contradiction to $\ell_T(\hat{e})=\pi(T)$.
Then, assume $\hat{e}'=\{r,s_h\}$ for some $h>k$.
The condition $h>k$ is due to $|V(B'_1)|=\cdots=|V(B'_k)|\geq|V(B_{k+1})|\geq\cdots\geq|V(B_d)|$.
By the same reason, we have $\ell_{T'}(\hat{e})=\ell_{T'}(\{r,s_1\})=\cdots\ell_{T'}(\{r,s_k\})\geq\ell_{T'}(\{r,s_h\})=\ell_{T'}(\hat{e}')$, a contradiction occurs.

By the induction hypothesis that $w(T')<\frac{3}{2}\pi(T)$, we have
\begin{align}
w(T') &< \frac{3}{2}\pi(T') = \frac{3}{2}(b_1-1)\big(1+(b_2-1)+\cdots+(b_k-1)+b_{k+1}+\cdots+b_d\big) \label{eq:T'-upper1} \\
&= \frac{3}{2}b_1(1+b_2+\cdots+b_d) - \frac{3}{2}\big((k-1)b_1+b_2+b_3+\cdots+b_d-(k-2)\big) \label{eq:T'-upper2} \\
&= \frac{3}{2}\pi(T) - 3(k-1)b_1 - \frac{3}{2}(b_{k+1}+b_{k+2}+\cdots+b_d) + \frac{3}{2}(k-2), \label{eq:T'-upper}
\end{align}
where the last equation is due to $b_1=b_2=\cdots=b_k$.
Let $R'$ be the all-to-all routing of $T'$.
As $w(T)\leq w(T')+\chi(Q(R\setminus R'))$, we consider the remaining paths in $R\setminus R'$.

When $k=1$, $R\setminus R'=\big\{P_{u_1,y}:\,y\in V(T)\setminus\{u_1\}\big\}$.
By assigning one new color to each path in $R\setminus R'$, it follows from \eqref{eq:T'-upper} and the assumption $b<a/2$ that
\begin{align*}
w(T) &\leq w(T') + b_1 + b_2 + \cdots + b_d \\
&< \frac{3}{2}\pi(T) - \frac{3}{2}(b_2+b_3+\cdots+b_d) - \frac{3}{2} + (b_1 + b_2 + \cdots + b_d) \\
&= \frac{3}{2}\pi(T) + b_1 - \frac{1}{2}(1+b_2+b_3+\cdots+b_d) - 1 \\
&< \frac{3}{2}\pi(T).
\end{align*}

When $k=2$, $R\setminus R'$ is divided into the following classes.
\begin{itemize}
\item $\mathcal{P}_1:=\big\{P_{x,y}:\,x=u_1,y\in V(B'_1)\cup\{r\} \text{ or } x=u_2,y\in V(B'_2)\cup\{r\}\big\}$.
\item $\mathcal{P}_2:=\{P_{x,y}:\,x=u_1,y\in V(B'_2) \text{ or } x=u_2,y\in V(B'_1) \text{ or } x=u_1,y=u_2 \big\}$.
\item $\mathcal{P}_i:=\{P_{x,y}:\,x=u_1,y\in V(B_i) \text{ or } x=u_2,y\in V(B_{i+1})\}$, for $i=3,4,\ldots, d-1$.
\item $\mathcal{P}_d:=\{P_{x,y}:\,x=u_1,y\in V(B_d) \text{ or } x=u_2,y\in V(B_3)\}$.
\end{itemize}
One can check that $\chi(Q(\mathcal{P}_1))=b_1$, $\chi(Q(\mathcal{P}_2))=|\mathcal{P}_2|=2b_1-1$, $\chi(Q(\mathcal{P}_d))=b_3$, and $\chi(Q(\mathcal{P}_i))=b_i$ for $3\leq i\leq d-1$.
It follows from \eqref{eq:T'-upper} that
\begin{align}
w(T) &\leq w(T') + \sum_{i=1}^d \chi(Q(\mathcal{P}_i)) \notag \\
&\leq w(T') +3b_1+2b_3+b_4+b_5+\cdots+b_{d-1}-1 \notag \\
&< \frac{3}{2}\pi(T) - 3b_1 - \frac{3}{2}(b_3+\cdots+b_d) + 3b_1+2b_3+b_4+b_5+\cdots+b_{d-1}-1 \notag \\
&= \frac{3}{2}\pi(T) + \frac{1}{2}(b_3-b_4-b_5-\cdots-b_{d-1}-3b_d)  - 1 \notag \\
&< \frac{3}{2}\pi(T), \label{eq:d>4-k=2}
\end{align}
where \eqref{eq:d>4-k=2} is due to $b_{d-1}+b_d>b_1$ and $d\geq 5$.

\smallskip
When $k\geq 3$ and $k$ is odd, $R\setminus R'$ is divided into the following classes.
\begin{itemize}
\item $\mathcal{P}_{(i,j)}:=\big\{P_{u_i,y}:\,y\in V(B'_j)\big\}$, for $1\leq i,j\leq k$.
\item $\mathcal{P}_0:=\big\{P_{u_i,r}:\,1\leq i \leq k\big\}$.
\item $\mathcal{P}_i:=\big\{P_{u_i,y}:\,y\in V(B_{k+1})\big\}$, for $1\leq i\leq k$.
\item $\mathcal{P}_\infty:=\big\{P_{u_i,u_j}:\,1\leq i\neq j \leq k\big\}$.
\item $\widehat{\mathcal{P}}_{(i,j)}:=\big\{P_{u_i,y}:\,y\in V(B_j)\big\}$, for $1\leq i\leq k$ and $k+2\leq j\leq d$.
\end{itemize}
Note that $\mathcal{P}_{(i,i)}$ refers to the collection of paths connecting $u_i$ and vertices in $B'_i$.
We remark here that $|\mathcal{P}_0|=k$, $|\mathcal{P}_\infty|={k\choose 2}$, and $|\mathcal{P}_{(i,j)}|=b_1-1$, $|\mathcal{P}_i|=b_{k+1}\leq b_1-1$, $|\widehat{\mathcal{P}}_{(i,j)}|=b_j$ for all suitable $i$ and $j$.

Recall that the total-chromatic number of $K_k$ is $k$ when $k$ is odd.
Let $K_k$ be a complete graph of $k$ vertices labelled $1,2,\ldots,k$, and let $f:V(K_k)\cup E(K_k)\to\{1,2,\ldots,k\}$ be a proper $k$-total-coloring of $K_{k}$.
For $t=1,2,\ldots,k$ denote by $C_t$ the collection of vertices and edges who receive color $t$ under $f$.
Notice that each $C_t$ contains exactly one vertex and $\frac{k-1}{2}$ edges.
For the sake of argument, we assume vertex $t\in C_t$.
For any $t$, construct two sets $\mathcal{O}_t$ and $\mathcal{O}^r_t$ as follows: For each edge $\{i,j\}\in C_t$, $i<j$, put $(i,j)$ and $(j,i)$ into $\mathcal{O}_t$ and $\mathcal{O}^r_t$, respectively.

Pick two paths, one in $\mathcal{P}_{(i,j)}$ and another in $\mathcal{P}_{(i',j')}$, they can receive the same color if $\{i,j\}\cap\{i',j'\}=\emptyset$.
For any $t$, since $t\notin\{i,j\}$ and $\{i,j\}\cap\{i',j'\}=\emptyset$ for any $(i,j),(i',j')\in\mathcal{O}_t$, we have
\begin{align*}
\chi\left( Q\left( \mathcal{P}_{(t,t)} \cup \bigcup_{(i,j)\in\mathcal{O}_t} \mathcal{P}_{(i,j)} \right) \right) \leq \max\left\{|\mathcal{P}_{(t,t)}|,|\mathcal{P}_{(i,j)}| \right\}_{(i,j)\in\mathcal{O}_t} = b_1-1.
\end{align*}
Since paths in $\mathcal{P}_t$ have no common edges with paths in $\mathcal{P}_{(j,i)}$, for any $(j,i)\in\mathcal{O}^r_t$, by the similar argument we have
\begin{align*}
\chi\left( Q\left( \mathcal{P}_t \cup \bigcup_{(j,i)\in\mathcal{O}^r_t} \mathcal{P}_{(j,i)} \right) \right) \leq \max\left\{|\mathcal{P}_t|,|\mathcal{P}_{(j,i)}| \right\}_{(j,i)\in\mathcal{O}^r_t} = b_1-1.
\end{align*}
By going through $t$ from $1$ up to $k$, it derives
\begin{equation}
\chi \left( Q \left( \bigcup_{1\leq t\leq k} \mathcal{P}_t \cup \bigcup_{1\leq i,j\leq k} \mathcal{P}_{(i,j)} \right) \right) \leq 2k(b_1-1).
\label{eq:d>4-k=odd_1}
\end{equation}

The paths in $\mathcal{P}_0\cup\mathcal{P}_\infty$ can be dealt with in the same way.
Any two paths in $$\big\{P_{t,r}\big\}\cup \big\{P_{u_i,u_j}:\,\{i,j\}\in C_t\big\}$$ have no common edges.
This implies that
\begin{equation}
\chi \left( Q \left( \mathcal{P}_0\cup\mathcal{P}_\infty \right) \right) = \chi \left( Q \left( \bigcup_{1\leq t\leq k} \left( \big\{P_{t,r}\big\}\cup \big\{P_{u_i,u_j}:\,\{i,j\}\in C_t\big\} \right) \right) \right)  \leq k.
\label{eq:d>4-k=odd_2}
\end{equation}

It remains to consider paths in $\widehat{\mathcal{P}}_{(i,j)}$, for $1\leq i \leq k$ and $k+2\leq j\leq d$.
Notice that any two paths, one in $\widehat{\mathcal{P}}_{(i,j)}$ and another in $\widehat{\mathcal{P}}_{(i',j')}$, have no common edges if and only if $i\neq i'$ and $j\neq j'$.
We only consider the case when $k\leq d-k-1$, since the other case (i.e., $k>d-k-1$) can be dealt with in the same way.
The classes of paths can be arranged in the following fashion.
\begin{center}
\begin{tabular}{c||ccccc}
$i$ & $S_i$ & & & \\ \hline
$1$ & $\widehat{\mathcal{P}}_{(1,k+2)}$, & $\widehat{\mathcal{P}}_{(2,k+3)}$, & $\widehat{\mathcal{P}}_{(3,k+4)}$, & $\ldots$ , & $\widehat{\mathcal{P}}_{(k,2k+1)}$. \\
$2$ & $\widehat{\mathcal{P}}_{(1,k+3)}$, & $\widehat{\mathcal{P}}_{(2,k+4)}$, & $\widehat{\mathcal{P}}_{(3,k+5)}$, & $\ldots$ , & $\widehat{\mathcal{P}}_{(k,2k+2)}$. \\
$\vdots$ & $\vdots$ & $\vdots$ & $\vdots$ & $\ddots$ & $\vdots$ \\
$d-k-1$ & $\widehat{\mathcal{P}}_{(1,d)}$\ \ , & $\widehat{\mathcal{P}}_{(2,k+2)}$, & $\widehat{\mathcal{P}}_{(3,k+3)}$, & $\ldots$ , & $\widehat{\mathcal{P}}_{(k,2k)}$.\ \ \ \ \\
\end{tabular}
\end{center}
In general, set $S_t$ collects the classes of paths $\widehat{\mathcal{P}}_{(1,t+k+1)},\widehat{\mathcal{P}}_{(2,t+k+2)},\widehat{\mathcal{P}}_{(3,t+k+3)},\ldots,\widehat{\mathcal{P}}_{(k,t+2k)}$, where the addition is taken modulo $d+1$ and plus $k+2$.
The chromatic number of the conflict graph induced by paths in $S_t$ is determined by the sizes of the classes $\widehat{\mathcal{P}}_{(i,j)}$ therein; more precisely, 
\begin{align*}
\chi(Q(S_t)) &\leq \max\left\{|\widehat{\mathcal{P}}_{(1,t+k+1)}|,|\widehat{\mathcal{P}}_{(2,t+k+2)}|,|\widehat{\mathcal{P}}_{(3,t+k+3)}|,\ldots,|\widehat{\mathcal{P}}_{(k,t+2k)}| \right\} \\
&= 
\begin{cases}
b_{k+2}, & \text{if } t=1 \text{ or } d-2k+1\leq t\leq d-k-1; \\
b_{t+k+1}, & \text{otherwise}.
\end{cases}
\end{align*}
Thus, we have
\begin{align}
\chi\left( Q\left( \biguplus_{\substack{1\leq i\leq k,\\ k+2\leq j\leq d}} \widehat{\mathcal{P}}_{(i,j)} \right) \right) &\leq \sum_{t=1}^{d-k-1} \chi\left( Q\left(S_t \right) \right) \notag \\
&= kb_{k+2}+b_{k+3}+b_{k+4}+\cdots+b_{d-k+1}. \label{eq:d>4-k=odd_3}
\end{align}

Combining \eqref{eq:d>4-k=odd_1}--\eqref{eq:d>4-k=odd_3}, we obtain
\begin{align*}
\chi(Q(R\setminus R')) \leq 2k(b_1-1) + k + kb_{k+2}+b_{k+3}+b_{k+4}+\cdots+b_{d-k+1}, 
\end{align*}
and further \eqref{eq:T'-upper} yields
\begin{align}
w(T) &\leq w(T') + \chi(Q(R\setminus R')) \notag \\
&< \frac{3}{2}\pi(T) - 3(k-1)b_1 - \frac{3}{2}\big(b_{k+1}+b_{k+2}+\cdots+b_d\big) + \frac{3}{2}(k-2) \notag \\
& \ \ \ \ \ + 2k(b_1-1) + k + kb_{k+2}+b_{k+3}+b_{k+4}+\cdots+b_{d-k+1} \notag \\
&= \frac{3}{2}\pi(T) -(k-3)b_1 - \frac{3}{2}\big(b_{k+1}\big) + \big(k-\frac{3}{2}\big)b_{k+2} + \frac{1}{2}(k-6) \notag \\
& \ \ \ \ \ - \frac{1}{2}\big(b_{k+3}+b_{k+4}+\cdots+b_{d-k+1}\big) - \frac{3}{2}\big( b_{d-k+2}+b_{d-k+3}+\cdots+b_d \big). \label{eq:d>4-k=odd_4}
\end{align}
Since $b_{k+2}\leq b_{k+1}\leq b_1-1$ and $k\geq 3$, one has 
\begin{equation}
-(k-3)b_1 - \frac{3}{2}\big(b_{k+1}\big) + \big(k-\frac{3}{2}\big)b_{k+2} + \frac{1}{2}(k-6)
\leq -(k-3) + \frac{1}{2}(k-6) < 0.
\label{eq:d>4-k=odd_5}
\end{equation}
Therefore, the result follows by plugging \eqref{eq:d>4-k=odd_5} into \eqref{eq:d>4-k=odd_4}.

\smallskip

When $k\geq 3$ and $k$ is even, the argument is similar to the odd case with a slight modification.
Let $T''$ be a tree obtained from $T'$ by removing an extra leaf $u_{k+1}$ from $B_{k+1}$, and let $B'_{k+1}=B_{k+1}-\{u_{k+1}\}$.
By the same argument in \eqref{eq:T'-upper1}--\eqref{eq:T'-upper}, we have
\begin{equation}
w(T'') < \frac{3}{2}\pi(T) - \frac{3}{2}(2k-1)b_1 - \frac{3}{2}\big( b_{k+1}+b_{k+2}+\cdots+b_d \big) + \frac{3}{2}(k-1).
\label{eq:T''-upper}
\end{equation}

Let $R''$ be the all-to-all routing of $T''$.
$R\setminus R''$ can be divided into the following classes.
\begin{itemize}
\item $\mathcal{P}_{(i,j)}:=\big\{P_{u_i,y}:\,y\in V(B'_j)\big\}$, for $1\leq i,j\leq k+1$.
\item $\mathcal{P}_0:=\big\{P_{u_i,r}:\,1\leq i \leq k+1\big\}$.
\item $\mathcal{P}_i:=\big\{P_{u_i,y}:\,y\in V(B'_{k+2})\big\}$, for $1\leq i\leq k+1$.
\item $\mathcal{P}_\infty:=\big\{P_{u_i,u_j}:\,1\leq i\neq j \leq k+1\big\}$.
\item $\widehat{\mathcal{P}}_{(i,j)}:=\big\{P_{u_i,y}:\,y\in V(B_j)\big\}$, for $1\leq i\leq k+1$ and $k+3\leq j\leq d$.
\end{itemize}
Observe that $k+1$ is even.
Again, by the same argument as proposed in \eqref{eq:d>4-k=odd_1}--\eqref{eq:d>4-k=odd_3}, we have
\begin{align}
\chi(Q(R\setminus R'')) \leq & 2(k+1)(b_1-1) + (k+1) + (k+1)b_{k+3} \notag \\ 
& +b_{k+4}+b_{k+5}+\cdots+b_{d-k}.
\label{eq:d>4-k=even_1}
\end{align}
Combining \eqref{eq:T''-upper} and \eqref{eq:d>4-k=even_1} yields
\begin{align}
w(T) &\leq w(T'') + \chi(Q(R\setminus R'')) \notag \\
&< \frac{3}{2}\pi(T) - \big(k-\frac{7}{2}\big)b_1 - \frac{3}{2}\big(b_{k+1}+b_{k+2}\big) + \big(k-\frac{1}{2}\big)b_{k+3} + \frac{1}{2}(k-5) \notag \\
& \ \ \ \ \ - \frac{1}{2}\big(b_{k+4}+b_{k+5}+\cdots+b_{d-k}\big) - \frac{3}{2}\big( b_{d-k+1}+b_{d-k+2}+\cdots+b_d \big). \label{eq:d>4-k=even_2}
\end{align}

It suffices to claim that 
\begin{equation}
- \big(k-\frac{7}{2}\big)b_1 - \frac{3}{2}\big(b_{k+1}+b_{k+2}\big) + \big(k-\frac{1}{2}\big)b_{k+3} + \frac{1}{2}(k-5) < 0.
\label{eq:d>4-k=even-claim}
\end{equation}
If $b_{k+3}=0$, the left-hand-side of \eqref{eq:d>4-k=even-claim} can be simplified as 
\begin{align*}
& -\big((k-4)+\frac{1}{2}\big)b_1 + \frac{1}{2}\big((k-4)-1\big) - \frac{3}{2}\big(b_{k+1}+b_{k+2}\big) \\
= & - \frac{1}{2}(k-4)(2b_1-1) - \frac{1}{2}(b_1+1) - \frac{3}{2}\big(b_{k+1}+b_{k+2}\big),
\end{align*}
which is less than $0$ due to $k\geq 4$.
If $b_{k+3}>0$, by plugging $b_{k+3}=b_1-\epsilon$, for some $\epsilon\geq 1$, into the left-hand-side of \eqref{eq:d>4-k=even-claim}, we derive from $b_{k+1}\geq b_{k+2}\geq b_{k+3}= b_1-\epsilon$ that
\begin{align*}
& - \big(k-\frac{7}{2}\big)b_1 - \frac{3}{2}\big(b_{k+1}+b_{k+2}\big) + \big(k-\frac{1}{2}\big)b_{k+3} + \frac{1}{2}(k-5) \\
\leq & -k\epsilon + \frac{1}{2}k + \frac{7}{2}\epsilon - \frac{5}{2} = -(\epsilon-1)\big(k-\frac{7}{2}\big) - \frac{1}{2}(k-2) \\
< & \ 0,
\end{align*}
as desired. 
\qed

\section{Concluding Remarks}\label{sec:conclusion}
In this paper, we prove that for any tree $T$, $w(T)<\frac{3}{2}\pi(T)$.
For any odd integer $k\geq 3$ and positive integer $t$, let $T_{k,t}$ be a rooted tree with $k$ branches, each of which has exactly $t$ vertices.
It has been showed in \cite{FLWZ17} that $w(T_{k,t})=kt^2$ and $\pi(T_{k,t})=(k-1)t^2+t$.
Therefore, $$\lim_{t\to\infty}\frac{w(T_{k,t})}{\pi(T_{k,t})}=\frac{k}{k-1},$$ which is upper-bounded by $\frac{3}{2}$ due to $k\geq 3$.
This indicates that the ratio $\frac{3}{2}$ is an asymptotically tight upper bound.
For a rational number $\delta$ between $1$ and $\frac{3}{2}$, it is called \emph{feasible} if there exists a tree $T$ such that $w(T)=\delta\pi(T)$.
It would be interesting to determine the spectrum of feasible rational numbers.
Some known feasible rational numbers can be found in~\cite{FLWZ17}.

\section{Declarations}\label{sec:declarations}
Data sharing not applicable to this article as no datasets were generated or analysed during the current study.

\section*{Acknowledgements}

The authors thank the referees for reading the manuscript carefully and providing helpful suggestions that improve the presentation of the paper.
This work was supported in part by the National Science and Technology Council, Taiwan, under grants 112-2115-M-153-004-MY2 and 104-2115-M-009-009, and the Fundamental Research Funds for the Central Universities of China under grant number 30920021127.

\rm

\end{document}